\documentclass[final,leqno,onetabnum]{siamltex704}

\usepackage{amsmath,amssymb} 
\usepackage{amsfonts} 
\usepackage{exscale}

\usepackage[%
  breaklinks=true,%
  colorlinks=true,%
  linkcolor=blue,anchorcolor=blue,%
  citecolor=blue,filecolor=blue,%
  menucolor=blue,%
  urlcolor=blue]{hyperref}
\usepackage{url,doi}

\renewcommand{\ldots}{\ensuremath{\dotsc}}
\newcommand{\comment}[1]{}

\def\H{{\mathcal H}}
\newcommand{\X}{\mathcal{X}}

\def\span{{\rm span}}
\def\dim{{{\rm dim}}}
\def\max{{{\rm max}}}
\def\min{{{\rm min}}}

\newtheorem{remark}[theorem]{Remark}

\title{Bounds for the Rayleigh quotient and the spectrum of self-adjoint operators
\thanks{%
\today .
Preliminary posted at \url{http://arxiv.org/abs/math/0610498}. 
This material is based upon work partially supported by the National Science Foundation under Grant No. 1115734. 
}}

\author{%
Peizhen Zhu\footnotemark[2]\ \footnotemark[3],
Merico E. Argentati\footnotemark[2]\ \footnotemark[3]
\and
Andrew V. Knyazev\footnotemark[2]\ \footnotemark[3]\ \footnotemark[4]\ \footnotemark[5]
}

\begin{document}
\vspace{-1.2in}
\vspace{.9in}

\setcounter{page}{1}
\maketitle

\renewcommand{\thefootnote}{\fnsymbol{footnote}}
\footnotetext[2]{Department of Mathematical and Statistical Sciences; 
University of Colorado Denver,
P.O. Box 173364, Campus Box 170, Denver, CO 80217-3364, USA.}
\footnotetext[3]{(peizhen.zhu,merico.argentati,andrew.knyazev)[at]ucdenver.edu}
\footnotetext[4]{Mitsubishi Electric Research Laboratories; 201 Broadway
Cambridge, MA 02139}
\footnotetext[5]{
 \url{http://www.merl.com/people/?user=knyazev} and \url{http://math.ucdenver.edu/~aknyazev/}
}
\renewcommand{\thefootnote}{\arabic{footnote}}

\begin{abstract}
The absolute change in the Rayleigh quotient (RQ) is bounded in this paper 
in terms of the norm of the residual and the change in the vector. 
If $x$ is an eigenvector of a self-adjoint bounded 
operator $A$ in a Hilbert space, then the RQ of the vector $x$, denoted by  $\rho(x)$, is an exact 
eigenvalue of $A$. In this case, the absolute change of the RQ 
$|\rho(x)-\rho(y)|$ becomes the absolute error for an eigenvalue  $\rho(x)$ of $A$ approximated 
by the RQ $\rho(y)$ on a given vector $y.$ There are
three traditional kinds of bounds for eigenvalue errors: 
a priori bounds via the angle between vectors $x$ and $y$; 
a posteriori bounds via the norm of the residual $Ay-\rho(y)y$ of vector $y$;
mixed type bounds using both the angle and the norm of the residual. 
We propose a unifying approach to prove known bounds of the spectrum,  
analyze their sharpness, and derive new sharper bounds. 
The proof approach is based on novel RQ vector perturbation identities.
\end{abstract}
\begin{AM}
15A42, 
15A60, 
65F35. 
\end{AM}


\pagestyle{myheadings}
\thispagestyle{plain}
\markboth{PEIZHEN ZHU, MERICO E. ARGENTATI, and ANDREW V. KNYAZEV}
{RQ perturbation bounds} 

\section{Introduction}\label{intro}

Let $A$ be a bounded self-adjoint operator in a real 
or complex Hilbert space $\H$.
For a nonzero vector $x$, the Rayleigh quotient (RQ) is defined by 
\[
\rho(x)=\rho(x,A)=
\frac{\left\langle x,Ax\right\rangle}{\left\langle x,x\right\rangle}
\]
and the corresponding residual vector is denoted by
$r\left(x\right)=r(x,A)=Ax-\rho(x)x,$
where $\left\langle \cdot,\cdot\right\rangle$ is an inner product, associated with a norm 
$\|\cdot\|^2=\left\langle \cdot, \cdot\right\rangle.$
The acute angle between two nonzero vectors $x$ and $y$ is denoted by 
\[
\angle\left\{x,y\right\}=\arccos\frac{|\left\langle x, y\right\rangle|}{\|x\|\|y\|}.
\]

We are interested in RQ vector perturbation bounds. 
Specifically, for nonzero vectors $x$ and $y$, 
we want to bound the following quantity, $|\rho(x)-\rho(y)|$,  
in terms of $\angle\left\{x,y\right\}$,  
$\|r(x)\|$, and $\|r(y)\|.$ If $x$ is an eigenvector of $A$, then 
the RQ of $x$ is an exact eigenvalue of $A$ and $r(x)=0.$ 
In this case, the absolute change of the RQ $|\rho(x)-\rho(y)|$ becomes 
the absolute error in the eigenvalue $\rho(x)$ of $A$. The RQ is 
often used to approximate points of the spectrum $\Sigma(A)$ of $A$.

Known bounds of the spectrum are traditionally classified depending on 
the terms that appear in the bounds. Let $\lambda\in\Sigma(A)$ be approximated by 
$\rho(y)$, i.e., the absolute approximation error is $|\lambda-\rho(y)|.$ 
Bounds of $|\lambda-\rho(y)|$ that are based on the norm of the residual 
$\|r(y)\|$ are called ``a posteriori'' bounds, since the residual 
$r(y)$ and its norm can typically be computed for the given vector $y$.

If $\lambda$ is an eigenvalue with the corresponding eigenvector $x$, 
bounds for $|\lambda-\rho(y)|$ that rely on the angle $\angle\left\{x,y\right\}$ are called 
``a priori,'' since the eigenvector $x$ is usually not explicitly known, and 
some a priori information about $x$ needs to be used to bound 
the angle $\angle\left\{x,y\right\}$. In the context of Finite Element Method 
error bounds, where $x$ and $y$ are functions, such information is usually 
associated with the smoothness of the function $x$; see, e.g.,\ \cite{ka09} and references there.
Finally, we call a bound  
``a mixed type'' if it involves both terms, the residual norm $\|r(y)\|$  and
the angle $\angle\left\{x,y\right\}$.

Sharp bounds of the spectrum approximation error are very important in the theory
of numerical solution of self-adjoint eigenvalue problems.
Computable a posteriori bounds allow one to obtain approximations from below and above to 
points of the spectrum. A~priori and mixed type bounds give one 
an opportunity to determine the quality of approximation of a point of the spectrum by using the RQ. 

We revisit this classical topic of research for three reasons. 
Our first goal is to carefully examine a posteriori and a priori bounds together, 
to understand what they have in common and why they are so different. 
Second, we are interested in discovering a general framework 
for deriving bounds of $|\rho(x)-\rho(y)|$ for arbitrary vectors 
$x$ and $y$ where known bounds of the spectrum approximation error become corollaries. 
Last, but not least, bounds of $|\rho(x)-\rho(y)|$ for arbitrary vectors 
$x$ and $y$  are practically important on their own, e.g.,\ 
as a tool for deriving sharp convergence rate bounds of eigenvalue solvers and 
analyzing the effects of inexact computations. 
We anticipate that our results will lead to new proof techniques and sharp bounds  
of accuracy of Ritz and harmonic (e.g.,\ \cite{se2003}) Ritz pairs, 
which are important in applications.

We start, in section \ref{review}, by reviewing three known bounds of 
the spectrum approximation error, each representing 
a priori, a posteriori, and mixed types, correspondingly. 
In section~\ref{section2}, we derive new identities and bounds for 
the change in the RQ with respect to the change of the vectors, using an orthogonal
projector on the subspace $\span\{x,y\}.$ The idea of such a ``low-dimensional'' analysis 
has been found fruitful before, e.g.,\ in \cite{k86}. We observe that in the case 
$\dim\H=2$ all three bounds, which we want to reinvent, turn 
into the same identity. 
In section \ref{section3}, we derive the a priori and mixed type bounds using this identity, 
using a single inequality, different for each bound, only at the last step. 
This unifying proof technique allows us to easily specify the circumstances where the bounds are sharp. 
We dedicate section~\ref{section4} to a posteriori bounds. Our ``low-dimensional'' 
analysis not only allows us to find a novel proof of a well known a posteriori bound, but also 
gives us an opportunity to improve it, and obtain several new sharper results. 
Our bonus section \ref{sec:vectors} touches the topic of improving known  $\sin(2\theta)$ and $\tan(\theta)$
error bounds for eigenvectors.

\section{Short review of some known error bounds for eigenvalues}\label{review}
Let $y$ be an approximation to an eigenvector $x$ of the operator $A$, corresponding to 
the eigenvalue $\rho(x)$ of $A.$ An a priori bound involves a constant 
and the square of the sine of the angle between the eigenvector $x$ and the vector $y$, 
see, e.g.,\ \cite{{kja10},{ka03},{ruhe76}},
\begin{equation}
\label{eq:sinsquare1} 
|\lambda-\rho(y)|\leq \left(\Sigma_{\max}(A)-\Sigma_{\min}(A)\right)\sin^2\left( \angle\left\{x,y\right\} \right),
\end{equation}
where we denote $\lambda=\rho(x)$, and 
$\Sigma_{\max}(A)$ and $\Sigma_{\min}(A)$
denote the largest and the smallest points of the spectrum $\Sigma(A)$ of $A$, correspondingly. 

The mixed type bound shown in 
\cite{{knyazev97},sun91} is in terms of the
 norm of the residual vector and the tangent of the angle between 
 vectors $x$ and $y$, i.e.,\
\begin{equation}
\label{eq:tanresidual1}
|\lambda-\rho(y)|\leq \frac{\|r(y)\|}{\|y\|}\tan\left(\angle\left\{x,y\right\}\right).
\end{equation}

There is no constraint on the location of the eigenvalue $\rho(x)=\lambda$ relative to $\rho(y)$
in \eqref{eq:sinsquare1} and \eqref{eq:tanresidual1}. 
Bounds \eqref{eq:sinsquare1} and \eqref{eq:tanresidual1} hold
for any nonzero vector $y$ and any eigenvector~$x.$ 
The next bound we review removes the eigenvector~$x$ from the picture, but instead 
requires having some information about the 
spectrum of $A$ in the neighborhood of $\rho(y)$. 
This a posteriori bound involves the square of the norm of the
residual vector and the gap between $\rho(y)$ and the spectrum  $\Sigma\left(A\right)$ of the operator $A$.  

Let real numbers $\alpha<\beta$ be such that $\Sigma\left(A\right)\cap\left(\alpha, \beta\right)=\emptyset$. 
Using the spectral decomposition of the self-adjoint bounded operator $A$, 
and the fact that the quadratic polynomial $(t-\alpha)(t-\beta)\geq0$ for any real $t$ outside of 
the interval $(\alpha,\beta)$, it is easy to see that
$
\left(A-\alpha I\right)(A-\beta I)\geq0,
$
where $I$ is the identity, i.e.,\ that 
$\left\langle Ay-\alpha y,Ay - \beta y\right\rangle\geq0$
for any vector $y$. Elementary algebraic manipulations for nonzero $y$ show that 
this inequality is equivalent to the famous bound attributed to Temple, see \cite{kato76,temple28}, 
\begin{equation}
\label{eqn:betaalphabound}
\left(\beta-\rho(y)\right)(\rho(y)-\alpha)\leq\frac{\|r(y)\|^2}{\|y\|^2}.
\end{equation}
In the nontrivial case $\alpha<\rho\left(y\right)<\beta$, we choose 
the scalars $\alpha\in\Sigma\left(A\right)$ and $\beta\in\Sigma\left(A\right)$ 
as the nearest points of the spectrum below and above $\rho\left(y\right).$
This results in the tightest bound \eqref{eqn:betaalphabound}, since
its left-hand side is monotonic in $\alpha$ and $\beta$.
In other words, for given $A$ and $y$ with $\rho(y)\notin\Sigma\left(A\right)$, 
there exist $\alpha\in\Sigma\left(A\right)$ and $\beta\in\Sigma\left(A\right)$ 
such that  $\alpha<\rho\left(y\right)<\beta$ and  
\eqref{eqn:betaalphabound} holds. 
Bound \eqref{eqn:betaalphabound} is sharp if 
$\alpha$ and $\beta$ are eigenvalues of $A$ and the vector $y$ is a linear combination of 
the corresponding eigenvectors. 

Next, for  $A$ and $y$ with $\rho(y)\notin\Sigma\left(A\right)$, 
there exist  $\lambda\in\Sigma\left(A\right)$ and real scalars $a$ and $b$  
such that  $a<\rho\left(y\right)<b$ and $a\leq\lambda\leq b,$ 
while  $\Sigma\left(A\right)\cap\left(a, \lambda\right)=\emptyset$
and $\Sigma\left(A\right)\cap\left(\lambda,b\right)=\emptyset$, and 
Kato-Temple's inequality (see, e.g.,\ \cite{{kato76}} and \cite[Theorem VIII.5, p. 84]{rs}) holds,
\begin{equation}
\label{eqn:katoinequality}
-\frac{1}{\rho(y)-a}\frac{\|r(y)\|^2}{\|y\|^2}\leq
\rho(y)-\lambda\leq\frac{1}{b-\rho(y)}\frac{\|r(y)\|^2}{\|y\|^2}.
\end{equation}
Indeed, in the case where $\lambda$ is located to the right of $\rho(y)$ 
we take $\alpha=a$ and $\beta=\lambda$ in \eqref{eqn:betaalphabound} to get the lower bound
of \eqref{eqn:katoinequality}. In the opposite case, $\lambda<\rho(y)$,
we take $\alpha=\lambda$ and $\beta=b$
to obtain the upper bound of \eqref{eqn:katoinequality}. 
If $\lambda\in(a,b)$, then $\lambda=\Sigma\left(A\right)\cap(a,b)$, 
and thus $\lambda$ is unique and in fact is an isolated eigenvalue.

Let $a\in\Sigma\left(A\right)$ and $b\in\Sigma\left(A\right)$ in \eqref{eqn:katoinequality} and denote
$\delta=\min_{\eta\in\{\Sigma(A)\backslash\{\lambda\}\}}|\eta-\rho(y)|$. 
Checking different scenarios $\lambda=a$, $\lambda=b$, and $\lambda\in(a,b)$  in \eqref{eqn:katoinequality}, 
we observe that inequalities~\eqref{eqn:katoinequality} imply, see, e.g.,\ \cite{{Chatelin},{saas92},{ss90}},
the existence of  $\lambda\in\Sigma\left(A\right)$ such that 
\begin{equation}
\label{eq:gapresidual1}
|\lambda-\rho(y)|\leq 
\frac{1}{\delta}\frac{\|r(y)\|^2}{\|y\|^2}.
\end{equation}
In bound \eqref{eq:gapresidual1}, $\lambda\in\Sigma\left(A\right)$ may be not unique and 
in general does not have to be an isolated point of the spectrum, or even an eigenvalue. 
Bound \eqref{eq:gapresidual1} is sharp if at least one 
of the scalars $\rho(y)-\delta$ or $\rho(y)+\delta$
is an eigenvalue of $A$ and  the vector $y$ is a linear combination of 
the corresponding eigenvector and an eigenvector corresponding to $\lambda$.

The Krylov-Weinstein \cite[p. 321]{{yo80}} (Bauer-Fike \cite[Theorem 3.6]{{saas92}}) theorem 
states the existence (but not necessarily uniqueness) of  $\lambda\in\Sigma\left(A\right)$ such that 
\begin{equation} \label{eq:residual1}
|\lambda-\rho(y)|\leq \frac{\|r(y)\|}{\|y\|}.
\end{equation}
It follows directly from \eqref{eqn:betaalphabound}, by choosing 
$\lambda=\alpha\in\Sigma\left(A\right)$  or $\lambda=\beta\in\Sigma\left(A\right)$,
whichever is closer to $\rho(y)$.
As in \eqref{eq:gapresidual1}, $\lambda\in\Sigma\left(A\right)$ does not have to be 
an isolated point of the spectrum, or even an eigenvalue. 
Bound \eqref{eq:residual1} turns into equality if 
bound \eqref{eqn:betaalphabound} is equality and, in addition, $\rho(y)-\alpha=\beta-\rho(y)$
in  \eqref{eqn:betaalphabound} .

This derivation of \eqref{eq:residual1} from bound \eqref{eqn:betaalphabound} 
following \cite[Corollary 6.20, p. 303]{Chatelin} is not so well known.  
We conclude that bound \eqref{eqn:betaalphabound} is the most fundamental a posteriori bound,
since all other a posteriori bounds, reviewed here, can be derived from it. 

The main goal of the rest of the paper is to revisit bounds \eqref{eq:sinsquare1},\,
\eqref{eq:tanresidual1}, and \eqref{eqn:betaalphabound}.
We propose a new unifying approach to prove them,  
analyze their sharpness, and derive some new sharper bounds. 
The proof approach is based on novel RQ vector perturbation identities 
presented in the next section.

\section{Key identities for the RQ}\label{section2}
 We start with a few trivial but key properties of the RQ. Then we derive a couple of 
simple expressions for the norm of the residual. Finally, our main results, 
several identities for the absolute change in the RQ, follow. 

Let $A$ be a bounded self-adjoint operator on a real or complex Hilbert space $\H$.
Let $S$ denote a subspace of $\H$ and $P_{S}$ be an orthogonal
projector on $S$. Let $A_S=\left(P_{S}A\right)|_{S}$ denote 
the restriction of the operator $P_{S}A$ to the subspace $S$. 
For a nonzero vector $x\in S$, we denote 
$\rho(x, A_S)=\langle x,A_S x\rangle/\langle x,x\rangle$ and
$r(x, A_S)=A_S x-\rho(x, A_S)x\in S$. We start with a couple of trivial, but
extremely important, for our approach, lemmas. 
\begin{lemma}
\label{lemma:rayleighquoientequation}
 If $x\in S$ then $\rho(x, A_S)=\rho\left(x\right)$.\end{lemma}
\begin{proof} The orthogonal projector $P_S$ is self-adjoint, so
\begin{eqnarray*} 
\rho(x, A_S)=\frac{\left\langle x,P_{S}Ax\right\rangle}{\left\langle x,x\right\rangle}=
\frac{\left\langle x,Ax\right\rangle}{\left\langle x,x\right\rangle}=\rho\left(x\right).
\end{eqnarray*} 
\end{proof}

\begin{lemma}
\label{lemma:residualequation}
 If $ x \in S$ then $r(x, A_S)=P_{S}r(x).$
\end{lemma}
\begin{proof}Directly by the definition and Lemma \ref{lemma:rayleighquoientequation}, we obtain
\begin{eqnarray*} 
r(x, A_S)=A_Sx-\rho(x, A_S)x=P_{S}Ax-\rho(x)x
                   =P_{S}(Ax-\rho(x)x)
                   =P_{S}r(x).
\end{eqnarray*} 
\end{proof}

\begin{corollary}
\label{cor:eigenvector}
If $x\in S$ is an eigenvector of $A$, it is also an eigenvector of $A_S$, 
corresponding to the same eigenvalue  $\rho(x)=\rho(x, A_S)$.
\end{corollary}
\begin{proof}
By Lemma  \ref{lemma:rayleighquoientequation}, $\rho(x)=\rho(x, A_S).$
For an eigenvector $x$ of $A$, the corresponding eigenvalue is $\rho(x)=\rho(x, A_S)$, so
$r(x)=0$, and $r(x, A_S)=P_{S}r(x)=0$, by Lemma  \ref{lemma:residualequation}.
Thus $x$ is also an eigenvector of $A_S$, with the same eigenvalue.
\end{proof}

{\bf
\emph{In the rest of the paper we always assume or prove that $\dim\,S=2.$}   
}

Now, we collect some basic identities, using the eigenvectors of $A_S$.
 Let us denote%
\footnote{We intentionally introduce new notation $\lambda_{\max}(A_S)=\Sigma_{\max}(A_S)$ and 
$\lambda_{\min}(A_S)=\Sigma_{\min}(A_S)$ to underline the fact that 
these points of the spectrum $\Sigma(A_S)$ are actually \emph{eigenvalues}.} 
$\mu=\lambda_{\max}(A_S)=\rho(u_1)$ and $\nu=\lambda_{\min}(A_S)=\rho(u_2)$, 
where $u_1$ and $u_2$ are orthogonal eigenvectors of the operator $A_S$. 
Let $P_i$ with $i=1,\,2$ denote the orthogonal projector on the subspace $\span\{u_i\}$. 
Assuming a nontrivial case $P_i x\neq 0$ for $i=1,\,2$, we evidently have  
$\angle\left\{x,u_i\right\}=\angle\left\{x,P_i x\right\}$. Now, since $x\in S$ and $\dim\,S=2,$  
\[
\rho\left(x\right)=\frac{\langle x,A_Sx\rangle}{\langle x,x\rangle}=
\frac{\langle P_1x,A_SP_1x\rangle+\langle P_2x,A_SP_2x\rangle}{\langle P_1x,P_1x\rangle+\langle P_2x,P_2x\rangle}=
\frac{\mu\langle P_1x,P_1x\rangle+\nu\langle P_2x,P_2x\rangle}{\langle P_1x,P_1x\rangle+\langle P_2x,P_2x\rangle}
\]
where $x=P_1x+P_2x$ and $P_1P_2=P_1A_SP_2=0$. Therefore, 
\begin{equation} \label{eqn:rem2D1}
\mu-\rho\left(x\right)=(\mu-\nu)\sin^2\left(\angle\left\{x,u_1\right\}\right) 
\text{ and } 
\rho\left(x\right)-\nu=(\mu-\nu)\sin^2\left(\angle\left\{x,u_2\right\}\right).
\end{equation}
Evidently, in a two-dimensional space,
\begin{eqnarray} \label{eqn:rem2D2}
 \cos\left(\angle\left\{x,u_1\right\}\right)=\sin\left(\angle\left\{x,u_2\right\}\right)  
\text{ and } 
\cos\left(\angle\left\{x,u_2\right\}\right)=\sin\left(\angle\left\{x,u_1\right\}\right).
\end{eqnarray}

\subsection{Identities for the norm of the residual $r\left(x,A_S\right)=P_{S}r\left(x\right)$} 
Our main identity is in the following lemma. 
\begin{lemma}
\label{lemma:generalresidual1}
Let $x\in S$ and $\dim\,S=2.$ Then 
\begin{eqnarray} \label{eqn:generalresidual1}
\left[\lambda_{\max}\left(A_S\right)-\rho(x)\right]
\left[\rho(x)-\lambda_{\min}\left(A_S\right)\right]
&=&\frac{\|P_{S}r\left(x\right)\|^2}{\|x\|^2}.
\end{eqnarray}
\end{lemma}
\begin{proof}
Let us denote $\mu=\lambda_{\max}(A_S)$ and $\nu=\lambda_{\min}(A_S)$.
By Lemmas \ref{lemma:rayleighquoientequation} and \ref{lemma:residualequation}, 
identity \eqref{eqn:generalresidual1} can be equivalently rewritten as
\begin{eqnarray*} 
\left[\mu-\rho(x,A_S)\right]
\left[\rho(x,A_S)-\nu\right]&=&\frac{\|r(x,A_S)\|^2}{\|x\|^2}.
\end{eqnarray*}
Since $\langle r(x,A_S),x\rangle=0$, we have $\|r(x,A_S)\|^2=\|A_S x\|^2-\rho^2(x,A_S)\|x\|^2$ and thus
\begin{eqnarray*}
\frac{\|r(x,A_S)\|^2}{\|x\|^2}-(\mu-\rho(x,A_S))(\rho(x,A_S)-\nu)&=& 
\frac{\|A_S x\|^2}{\|x\|^2}-(\mu+\nu)\rho(x,A_S)+\mu\nu\\
&=&
\frac{\left\langle A_S x-\mu x,A_Sx-\nu x\right\rangle}{\|x\|^2}\\
&=&0,
\end{eqnarray*}
where $(A_S -\mu)(A_S-\nu)=0$ is the minimal polynomial of $A_S$, 
since $\dim\,S=2$ and $\mu$ and $\nu$ are the eigenvalues of $A_S.$
\end{proof}

We also mention the following identity, relying on eigenvectors $u_i$ of $A_S$. 
\begin{lemma}
\label{lemma:generalresidual2}
Let $x\in S$ and $\dim\,S=2$. 
Then 
\begin{eqnarray} \label{eqn:generalresidual2}
\frac{1}{2}\left[\lambda_{\max}\left(A_S\right)-\lambda_{\min}\left(A_S\right)\right]
\sin\left(2\angle\left\{x,u_i\right\}\right)
&=&\frac{\|P_{S}r\left(x\right)\|}{\|x\|},\, i=1,\,2.
\end{eqnarray}
\end{lemma}
\begin{proof}
Applying \eqref{eqn:rem2D1} in \eqref{eqn:generalresidual1}, and taking the square root, we get 
\[
(\mu-\nu)\sin\left(\angle\left\{x,u_1\right\}\right)\sin\left(\angle\left\{x,u_2\right\}\right)=
\frac{\|P_{S}r\left(x\right)\|}{\|x\|}.
\]
By \eqref{eqn:rem2D2}, $\sin\left(\angle\left\{x,u_2\right\}\right)=\cos\left(\angle\left\{x,u_1\right\}\right)$, so 
the statement of the lemma for $i=1$ follows from the trigonometry identity for the sine of a double angle. 
Identities \eqref{eqn:rem2D2} also imply $\sin\left(2\angle\left\{x,u_1\right\}\right)=\sin\left(2\angle\left\{x,u_2\right\}\right)$, 
even though it looks counterintuitive. 
\end{proof}

\subsection{Identities for the absolute change in the RQ}
Here we prove our main tangent- and sine-based identities. 

{\bf
\emph{In the remainder of the paper, for given linearly independent vectors $x$ and $y$ in $\H$ we always 
define the subspace $S$ as $S=\span\{x, y\}$.}  
}

\begin{remark} \label{rem:i} 
Let $0<\angle\left\{x,y\right\}< \pi/2$. We note the following useful identities 
\[
\frac{\|r(x)\|}{\|x\|}\frac{\cos\left(\angle\left\{r(x),y\right\}\right)}
{\cos\left(\angle\left\{x,y\right\}\right)}=
\frac{|\left\langle r(x),y\right\rangle|}
{|\left\langle x,y\right\rangle|}=
\frac{|\left\langle P_S r(x),y\right\rangle|}
{|\left\langle x,y\right\rangle|}=
\frac{\|P_S r(x)\|}{\|x\|}\frac{\cos\left(\angle\left\{P_S r(x),y\right\}\right)}
{\cos\left(\angle\left\{x,y\right\}\right)}
\]
and
\[
\frac{\|P_S r(x)\|}{\|x\|}\frac{\cos\left(\angle\left\{P_S r(x),y\right\}\right)}
{\cos\left(\angle\left\{x,y\right\}\right)}=
\frac{\|P_{S}r\left(x\right)\|}{\|x\|}
\tan\left(\angle\left\{x,y\right\}\right),
\]
where
$\cos\left(\angle\left\{P_S r(x),y\right\}\right)=\sin\left(\angle\left\{x,y\right\}\right)$.
Indeed, $0=\left\langle r(x),x \right\rangle= \left\langle P_S r(x),x \right\rangle$, 
i.e.,\ vectors $P_S r(x)\in S$ and $x\in S$ are orthogonal in $S$, and also $y\in S,$ where
$\dim\,S=2.$ 
Denoting the orthogonal projector on the subspace $\span\{y\}$ by $P_y$ and using $\angle\left\{P_y r(x),y\right\}=0,$ we also get  
\[
 \frac{|\left\langle r(x),y\right\rangle|}
{|\left\langle x,y\right\rangle|}=
\frac{|\left\langle P_y r(x),y\right\rangle|}
{|\left\langle x,y\right\rangle|}=
\frac{\|P_y r(x)\|}{\|x\|}\frac{\cos\left(\angle\left\{P_y r(x),y\right\}\right)}
{\cos\left(\angle\left\{x,y\right\}\right)}=
\frac{\|P_y r(x)\|}{\|x\|\cos\left(\angle\left\{x,y\right\}\right)}.
\]
\end{remark}
 
\begin{theorem}
\label{lemma:generalsumtan}
Let $0<\angle\left\{x,y\right\}< \pi/2$ 
and let us denote 
\[
\Xi_\pm=
\left|\frac{\|P_{S}r\left(x\right)\|}{\|x\|}\pm
\frac{\|P_{S}r\left(y\right)\|}{\|y\|}\right|
\tan\left(\angle\left\{x,y\right\}\right).
\]
We always have 
\begin{equation}\label{eq:t37}
\Xi_{-}\leq |\rho(x)-\rho(y)| \leq \Xi_{+}.
\end{equation}
Now, let us denote 
$a=\left\langle x,r(y)\right\rangle,\, b=\left\langle r(x),y\right\rangle,$
and $c=\left\langle x,y\right\rangle\neq0$ and 
assume that 
$a/c$ or $b/c$ 
is real. Then
$
|\rho(x)-\rho(y)|=\Xi_{-}
$ 
if 
$b\neq0$ and $a/b\geq0$ (or $a\neq0$ and $b/a\geq0$ );
otherwise,  
$
 |\rho(x)-\rho(y)|=\Xi_{+}.
$
\end{theorem}
\begin{proof}
We consider only the nontrivial case where at least one of the vectors $x$ and $y$ 
is not an eigenvector of $A_S$.  Let it be, e.g.,\ $y$, so that the set $y$ and $P_Sr(y)\neq0$ form an orthogonal basis 
of the subspace $S$. Then  $a=\left\langle x,r(y)\right\rangle=\left\langle x,P_Sr(y)\right\rangle\neq0$, 
since $\angle\left\{x,y\right\}\neq0$. 
By elementary calculations, we obtain
 \begin{eqnarray*} 
 \left(\rho(x)-\rho(y)\right)\left\langle x,y\right\rangle&=&
\left\langle \rho(x) x,y\right\rangle-\left\langle x,\rho(y)y\right\rangle
+\left\langle x,Ay\right\rangle-\left\langle Ax,y\right\rangle\\
&=&\left\langle x,r(y)\right\rangle-\left\langle r(x),y\right\rangle,
  \end{eqnarray*}
thus $\rho(x)-\rho(y)=(a-b)/c$ and 
$|\rho(x)-\rho(y)|=|a-b|/|c|$.
In a complex space, the RQ remains real, since $A$ is Hermitian, but the 
scalar products $a,\,b,$ and $c$ may be not real, so in general we have, 
by the triangle inequality for complex scalars, that
$\big|\,|a|-|b|\big|\leq|a-b|\leq|a|+|b|$,
where the inequalities are strict unless $b/a$ is a real number. 
On the other hand,  $\rho(x)-\rho(y)=(a-b)/c$ is always real, so 
 $b/a$ is real iff $a/c$ or $b/c$ is real, which is the second assumption of the theorem. 
Under this assumption, 
$\big|\,|a|-|b|\big|=|a-b|$, if $b/a\geq0$, or  
$|a-b|=|a|+|b|$,  if $b/a\leq0$,
which completely characterizes the cases where inequalities in \eqref{eq:t37} turn into equalities. 

The statements of the theorem now follow directly from Remark \ref{rem:i}.
\end{proof}

\begin{lemma}
Let  $\angle\{x,y\}>0$.
For both $i=1$ and $2$, we have 
\begin{equation}\label{eqn:sin1}
|\rho(x)-\rho(y)|=
[\mu-\nu]
\sin\left(\angle\left\{x,u_i\right\}+\angle\left\{y,u_i\right\}\right)\left|\sin\left(\angle\left\{x,u_i\right\}-\angle\left\{y,u_i\right\}\right)\right|.
\end{equation}
\end{lemma}
\begin{proof}
Identities \eqref{eqn:rem2D1} and \eqref{eqn:rem2D2} 
imply
\begin{eqnarray*}
|\rho\left(x\right)-\rho\left(y\right)|&=&
[\mu-\nu]\left|\cos^2\left(\angle\left\{x,u_i\right\}\right)-\cos^2\left(\angle\left\{y,u_i\right\}\right)\right|,
\end{eqnarray*}
which leads directly to identity \eqref{eqn:sin1} using elementary trigonometry.
\end{proof}

\begin{theorem}
\label{lemma:generalsine}
 Let $\angle\{x,y\}>0$.
Let us for $i=1$ or $2$ denote 
\[
\Psi_\pm=
\left[\lambda_{\max}\left(A_S\right)-\lambda_{\min}\left(A_S\right)\right]
\left|\sin\left(\angle\left\{x,u_i\right\}\pm\angle\left\{y,u_i\right\}\right)\right|
\sin\left(\angle\left\{x,y\right\}\right),
\]
and let 
$C=\left\langle x, u_1\right\rangle \left\langle u_2, x\right\rangle \left\langle u_1, y\right\rangle \left\langle y, u_2\right\rangle$.
We always have
$
\Psi_{-}\leq |\rho(x)-\rho(y)| \leq \Psi_{+}.
$ 
If the scalar $C$ is not real then both inequalities on the left and right are strict.
Moreover, if $C\ge0$ ($C\le0$) then we have equality for the lower (upper) bound.
\end{theorem}
\begin{proof}
 Identity \eqref{eqn:sin1} makes the statement of the theorem equivalent to 
\begin{equation} \label{eqn:r2}
|\sin\left(\angle\left\{x,u_i\right\}-\angle\left\{y,u_i\right\}\right)|\leq
\sin\angle\left\{x,y\right\}\leq\sin\left(\angle\left\{x,u_i\right\}+\angle\left\{y,u_i\right\}\right).
\end{equation}
For simplicity of notation, let us assume, without loss of generality, 
that all vectors involved are normalized, i.e.,\ $\|u_1\|=\|u_2\|=\|x\|=\|y\|=1$.
Using the representations of the vectors $x$ and $y$ with respect to the orthonormal basis $u_1$ and $u_2$ for the subspace $S$, we obtain by 
direct computation 
\[
\sin^2\angle\left\{x,y\right\}=\cos^2\angle\left\{x,u_1\right\} \cos^2\angle\left\{y,u_2\right\}+
\cos^2\angle\left\{x,u_2\right\} \cos^2\angle\left\{y,u_1\right\}-2\Re(C)
\]
and 
\[
|C|=\cos\angle\left\{x,u_1\right\} \cos\angle\left\{x,u_2\right\}
\cos\angle\left\{y,u_1\right\} \cos\angle\left\{y,u_2\right\}.
\]
Evidently, $-|C| \le \Re(C) \le |C|$, and if $C\ge0$ ($C\le0$) then $\Re(C) = |C|$ ($\Re(C) = -|C|$), 
which proves bounds \eqref{eqn:r2} and also gives conditions for their sharpness. 
\end{proof}
\begin{remark}
An alternative proof of Theorem \ref{lemma:generalsine} is based on the fact that 
angles between subspaces describe a metric, see, e.g.\ \cite{angles2010} and references there, so
\begin{equation}\label{eqn:sin2}
\left|\angle\left\{x,u\right\}-\angle\left\{y,u\right\}\right|\leq
\angle\left\{x,y\right\}\leq\left|\angle\left\{x,u\right\}+\angle\left\{y,u\right\}\right|.
\end{equation}  
\end{remark}

Theorem \ref{lemma:generalsine} trivially implies the bound 
\[
|\rho(x)-\rho(y)| \leq  
\left[\lambda_{\max}\left(A_S\right)-\lambda_{\min}\left(A_S\right)\right]
\sin\left(\angle\left\{x,y\right\}\right)
\]
that can be found for the real case in \cite{ka03}.

Finally, we apply our results above to an important special case.
\begin{corollary}\label{lemma:sinesquaren}
Let $x$ be an eigenvector of $A$ and $\rho(x)=\lambda$.  Then
\begin{eqnarray} \label{eqn:sine2}
|\lambda-\rho(y)|&=&\left[\lambda_{\max}\left(A_S\right)-
\lambda_{\min}\left(A_S\right)\right]
\sin^2\left(\angle\left\{x,y\right\}\right).
\end{eqnarray}
If, in addition $\angle\left\{x,y\right\}< \pi/2$, then
\begin{eqnarray} \label{eqn:tan}
|\lambda-\rho(y)|&=&\frac{\|P_{S}r\left(y\right)\|}{\|y\|}
\tan\left(\angle\left\{x,y\right\}\right).
\end{eqnarray}
If also $\eta\neq\lambda$ denotes an eigenvalue of $A_S$, then
\begin{eqnarray} \label{eqn:tanres}
\tan\left(\angle\left\{x,y\right\}\right)&=&\frac1{\left|\eta-\rho(y)\right|}
\frac{\|P_{S}r\left(y\right)\|}{\|y\|}.
\end{eqnarray}
\end{corollary}
\begin{proof}
Since $x$ is an eigenvector of $A$ it is also an eigenvector of $A_S$, 
with the same eigenvalue $\rho(x)=\lambda$, by Corollary \ref{cor:eigenvector}.
The statements \eqref{eqn:sine2} and \eqref{eqn:tan} follow directly 
from Theorems \ref{lemma:generalsumtan} and \ref{lemma:generalsine}, correspondingly. 
Identity \eqref{eqn:tanres} is derived from \eqref{eqn:generalresidual1} and \eqref{eqn:tan}.
Finally, the assumption $\angle\{x,y\}>0$ is dropped, since in the case $\angle\{x,y\}=0$
all the statements trivially hold, where both sides vanish. 
\end{proof}

The next section is entirely based on Corollary \ref{lemma:sinesquaren}. 
In the rest of the paper, we do not use our identities and bounds for $|\rho(x)-\rho(y)|$ for arbitrary 
nonzero vectors $x$ and $y$. 
We want to highlight again, that these results have merit on their own, 
not just as a tool for derivation of eigenvalue error bounds. 


\section{Deriving some known eigenvalue error bounds using our identities}\label{section3}
We now easily derive two known bounds, reviewed in section \ref{review}, 
from the results obtained in section \ref{section2}. 
First, since $\Sigma_{\min}\left(A\right)\leq
\lambda_{\min}\left(A_S\right)\leq\lambda_{\max}\left(A_S\right)\leq\Sigma_{\max}\left(A\right)$, 
from \eqref{eqn:sine2} we obtain a priori bound \eqref{eq:sinsquare1},  
\begin{eqnarray*} 
|\lambda-\rho(y)|&=&\left[\lambda_{\max}\left(A_S\right)
-\lambda_{\min}\left(A_S\right)\right]
\sin^2\left(\angle\left\{x,y\right\}\right)\\
                 &\leq&\left[\Sigma_{\max}(A)-\Sigma_{\min}(A)\right]
                 \sin^2\left(\angle\left\{x,y\right\}\right).
\end{eqnarray*}
The inequality becomes an equality if $\Sigma_{\max}(A)$ and $\Sigma_{\min}(A)$ are 
eigenvalues of $A_S$. For details, see reference \cite{ka03}.

The second known result, mixed bound \eqref{eq:tanresidual1}, follows from \eqref{eqn:tan}, 
\begin{eqnarray*} 
|\lambda-\rho(y)|&=&\frac{\|P_{S}r\left(y\right)\|}{\|y\|}
\tan\left(\angle\left\{x,y\right\}\right)\\
                 &\leq&\frac{\|r\left(y\right)\|}{\|y\|}
                 \tan\left(\angle\left\{x,y\right\}\right).
\end{eqnarray*}
The inequality turns into an equality iff $P_S r(y)=r(y)$, i.e.,\ iff
$Ay\in\span\{x, y\}$, which is equivalent to the subspace $\span\{x, y\}$
being $A$-invariant. 
More detailed information about the quality of this bound is determined by 
Remark~\ref{rem:i} where vectors $x$ and $y$ are swapped to give us the following identity
\[
{\|r(y)\|}{\cos\left(\angle\left\{r(y),x\right\}\right)}=
{\|P_S r(y)\|}{\cos\left(\angle\left\{P_S r(y),x\right\}\right)}.
\]

\begin{remark}\label{rem:unbounded}
 Our assumption that the operator $A$ is bounded is necessary in bound \eqref{eq:sinsquare1} and 
can be used in bound \eqref{eq:tanresidual1} to guarantee the existence of the vector $Ay$.  
Our results of  section \ref{section2} formally speaking would hold 
for an unbounded operator $A$ and given vectors $x$ and $y$, as soon as 
the operator $A_S$ can be correctly defined by $A_S=\left(P_SA\right)|_S$, where $S=\span\{x,y\}$.
One could do even better than that, in fact, one only needs to be able to correctly define the 
values $\rho(x)$ and $\rho(y)$. 

Rather than getting into technical details, we illustrate such a possibility 
using the following example from Davis and Kahan \cite{dk70}.
Let $\epsilon=1/2$ and $\H=l_2.$ We take 
$y=(1,\epsilon,\epsilon^2,\ldots)^T\in l_2$ and 
$A=\diag(1,\epsilon^{-1},\epsilon^{-2},\ldots)$.
We obtain $\rho(y)=1+\epsilon$ even though the  
sequence $Ay=(1,1,1,\ldots)^T$ has an infinite norm in $l_2$. 

Let us now consider $x=(1,0,0,\ldots)^T\in l_2$, the eigenvector
corresponding to the lowest eigenvalue $\rho(x)=1$, and 
define $w=y-x=(0,\epsilon,\epsilon^2,\ldots)^T\in l_2$
with $\rho(w)=(1+\epsilon)/\epsilon=3$. Then $A_S$ can be constructed  
using its eigenvectors $u_1=w$ and $u_2=x$ and the corresponding eigenvalues 
$\mu=\rho(w)=3$ and  $\nu=\rho(x)=1$ despite of the fact that the sequence 
$Aw=(0,1,1,\ldots)^T$ does not exist in $\H=l_2$. 

In the a priori bound, the eigenvalues of the operator $A_S$ are $1$ and $3$, 
and our result  \eqref{eqn:sine2} holds, while the bound \eqref{eq:sinsquare1} fails since 
$\Sigma_{\max}(A)=\infty.$ In the mixed bound, we have $P_S r(y)\in l_2$ in our \eqref{eqn:tan}, but 
$\|r\left(y\right)\|=\infty$ in \eqref{eq:tanresidual1}.
For details, see the end of the next section where we return to this example. 
\end{remark}

We have now demonstrated how two known bounds, a priori bound \eqref{eq:sinsquare1} and 
mixed bound \eqref{eq:tanresidual1}, can be derived in a simple and uniform manner 
from essentially two dimensional identities \eqref{eqn:sine2} and  \eqref{eqn:tan}.
The last well known bounds, we review in section \ref{review}, 
the a posteriori Temple bound \eqref{eqn:betaalphabound} and its follow-ups,
can also be derived in a similar way from their two dimensional prototype \eqref{eqn:generalresidual1}. 
However, the derivation is not so trivial, so we dedicate a separate section for it, coming next. 

\section{New a posteriori bounds for eigenvalues}\label{section4} For convenience, 
let us remind the reader of the Temple bound \eqref{eqn:betaalphabound},  
\begin{equation*} 
\left(\beta-\rho(y)\right)(\rho(y)-\alpha)\leq\frac{\|r(y)\|^2}{\|y\|^2},
\end{equation*}
where 
$\Sigma\left(A\right)\cap\left(\alpha, \beta\right)=\emptyset$, 
and its two dimensional analog \eqref{eqn:generalresidual1}, with $y$ substituted for $x$ for consistency of notation,
\begin{eqnarray*} 
\left[\lambda_{\max}\left(A_S\right)-\rho(y)\right]
\left[\rho(y)-\lambda_{\min}\left(A_S\right)\right]
&=&\frac{\|P_{S}r\left(y\right)\|^2}{\|y\|^2}.
\end{eqnarray*}
Our primary goal is to derive \eqref{eqn:betaalphabound} from \eqref{eqn:generalresidual1}. 
We first cover an important particular case, where we can even improve \eqref{eqn:betaalphabound}.
\begin{lemma} 
\label{lemma:projectorresidualsquare1}
Let $\alpha=\lambda_{\min}\left(A\right)$ be an isolated point of the spectrum 
of $A$ and $\beta>\alpha$ be the nearest larger point of the spectrum 
of $A$. Let $\X$ denote the complete eigenspace corresponding to 
the eigenvalue $\alpha,$ and $P_{\X}$ be the orthogonal projector onto $\X.$
For a vector $y\neq0$, such that $\alpha<\rho(y)<\beta$, we define 
$x=P_{\X}y$ and $S=\span\{x,y\}$. Then
 \begin{eqnarray}
 \label{eqn:largestresidualsqu} 
\left(\beta-\rho(y)\right)\left(\rho(y)-\alpha\right)&\leq&
\frac{\|P_{S}r(y)\|^2}{\|y\|^2},
\end{eqnarray}
where $P_{S}$ is the orthogonal projector on the subspace $S$.
\end{lemma}
\begin{proof}
The assumption
$\rho(y)>\alpha$ implies $y\notin\X$ since $\alpha=\lambda_{\min}\left(A\right)$, so $x\neq y$.
If $x=0$ then $y\perp \X$, but since  $\X$ is the \emph{complete} eigenspace corresponding to 
the eigenvalue $\alpha,$ and  $\beta>\alpha$ is the nearest larger point of the spectrum,  
$y\perp \X$ implies $\rho(y)\geq\beta$, which contradicts the lemma assumption $\rho(y)<\beta$. 
Thus $\dim\,S=2$ since $0\neq x=P_{\X}y\neq y$, so our results of section \ref{section2} hold. 

The vector $x$ is an eigenvector of $A_S$, and 
$\rho(x)=\lambda_{\min}\left(A\right)=\lambda_{\min}\left(A_S\right)$. 
Let $\X^\perp$ denote the orthogonal complement to 
$\X$.
The vector $w=y-P_{\X}y\in S\cap\X^{\perp}$ is an eigenvector corresponding to the other eigenvalue of $A_S$, i.e.,\
$\rho(w)=\lambda_{\max}\left(A_S\right)$.
Since $\X$ is the complete eigenspace corresponding to 
the eigenvalue $\alpha=\lambda_{\min}\left(A\right)$, the fact that $w\in \X^{\perp}$ 
guarantees that $\beta\leq\rho(w)$. Therefore, collecting the bounds together, we get 
$ 
 \lambda_{\min}\left(A_S\right)=\alpha<\rho(y)<\beta\leq\rho(w)=\lambda_{\max}\left(A_S\right).
$ 
Substituting $\beta\leq\lambda_{\max}\left(A_S\right)$ for $\lambda_{\max}\left(A_S\right)$ 
in identity \eqref{eqn:generalresidual1} gives the desired bound.
\end{proof}

The Temple bound \eqref{eqn:betaalphabound}, for the case of the smallest eigenvalue, 
follows from \eqref{eqn:largestresidualsqu}, 
using $\|P_{S}r\left(y\right)\|\leq\|r\left(y\right)\|$. 
In Remark \ref{rem:unbounded} we have already discussed that having a bound based on 
 $\|P_{S}r\left(y\right)\|$ rather than on $\|r\left(y\right)\|$ may provide great advantages. 
At the end of this section we extend the arguments of Remark \ref{rem:unbounded} 
to investigate it. 

Of course, Lemma \ref{lemma:projectorresidualsquare1} can be easily reformulated for the case
of the opposite side of the spectrum of the operator $A.$ 
Returning our attention to our main goal, we can construct simple examples showing that bound 
\eqref{eqn:largestresidualsqu} does not hold if neither $\alpha$ nor $\beta$
corresponds to the extreme points of the spectrum of the operator $A$. 
The extension of bound \eqref{eqn:largestresidualsqu} 
to the general case is described below. We start with a technical lemma.
\begin{lemma} 
\label{lemma:technical}
Let $\rho(y)\notin\Sigma(A)$ for a given vector $y\neq0.$ 
Let us denote by $U$ 
the invariant subspace of $A$ associated with all of the elements of the
spectrum of $A$ larger than $\rho(y)$; and let us define 
\[V=U+\span\{y\}.\]
Then $\Sigma\left(A_{V}\right)= \rho(v) \cup \Sigma\left(A_{U}\right)$,
and the smallest point  $\Sigma_{\min}\left(A_{V}\right)<\rho(y)$ 
of the spectrum $\Sigma\left(A_{V}\right)$ is an isolated point---an~eigenvalue of multiplicity one. 
\end{lemma}
\begin{proof}
 Since $V=U+\span\{y\}$, we have $\dim\left(V\cap U^\perp\right)=1.$ We define a nonzero vector
$v=(I-P_U)y\in V\cap U^\perp$, and notice that $V=U\oplus\span\{v\}$ is an orthogonal sum. 
Since $U$ is an $A$-invariant subspace, it is also $A_V$-invariant. Then  the decomposition 
$V=U\oplus\span\{v\}$ implies that $v$ is an eigenvector of the operator $A_V$,
corresponding to the eigenvalue $\rho(v).$
Moreover, by the variational principle for the Rayleigh quotient, 
$\rho(v)\leq\rho(y)<\Sigma_{\min}\left(A_{U}\right)$,
according to the definition of $U$. 
Thus, we deduce that 
$\Sigma\left(A_{V}\right)= \rho(v) \cup \Sigma\left(A_{U}\right)$, 
where  $\rho(v)=\Sigma_{\min}\left(A_{V}\right)$ is an isolated point.
The multiplicity of the eigenvalue  $\rho(v)$ is one, because, again,
$\dim\left(V\cap U^\perp\right)=1.$ 
\end{proof}

We are now prepared to prove our main result in this section. 
\begin{theorem} 
\label{lemma:projectorresidualsqu}
Let $\alpha<\rho(y)<\beta$ where $\alpha$ and $\beta$ are the nearest points of 
the spectrum of $A$ to $\rho(y)$. 
Let $U$ be the invariant subspace of $A$ associated with all of the  elements of the
spectrum of $A$ larger than $\rho(y)$. We define $x=(I-P_U)y$ and 
${S}=\span\{x, y\}$.
Then
\begin{eqnarray} 
 \label{eqn:genearalproresidualsqu}
\left(\beta-\rho(y)\right)\left(\rho(y)-\alpha\right)
&\leq&
\frac{\|P_{S}r(y)\|^2}{\|y\|^2}.
\end{eqnarray}
\end{theorem}
\begin{proof} 
Applying Lemma \ref{lemma:technical} and   
noticing that $\Sigma_{\min}\left(A_{U}\right)=\beta$, we observe that
the assumptions of Lemma \ref{lemma:projectorresidualsquare1} are satisfied 
if $A_{V}$ replaces $A$, where the scalar $\beta$ and the eigenvalue $\Sigma_{\min}\left(A_{V}\right)$ are 
the nearest points of the spectrum of the operator $A_{V}$ to  $\rho(y,A_{V})=\rho(y)\in\left(\Sigma_{\min}\left(A_{V}\right),\beta\right).$
Thus, by Lemma \ref{lemma:projectorresidualsquare1},
\[ 
\left(\beta-\rho(y)\right)\left(\rho(y)-\Sigma_{\min}\left(A_{V}\right)\right)
\leq\frac{\|P_{S}r(y,A_{V})\|^2}{\|y\|^2}
= \frac{\|P_{S}r(y)\|^2}{\|y\|^2}.
\]
Since  $\Sigma_{\min}\left(A_{V}\right)\leq\alpha<\rho(y)$, we obtain
\eqref{eqn:genearalproresidualsqu}.
\end{proof}
\begin{remark}
Substituting $-A$ for $A$ in Lemmas  \ref{lemma:projectorresidualsquare1} and \ref{lemma:technical},
and using similar arguments, we end up with the same subspace 
 $S=\span\{P_U y, y\}=\span\{(I-P_U)y, y\}$ and obtain exactly the same bound. 
\end{remark}

Repeating the arguments of section \ref{review}, 
we obtain improved versions of the Kato-Temple and other bounds, reviewed in section \ref{review},
from Theorem \ref{lemma:projectorresidualsqu}.
Specifically, defining the subspace $S$ as in Theorem \ref{lemma:projectorresidualsqu}
and using the notation and assumptions of section \ref{review}, we get 
the improved  Kato-Temple bound as follows, 
\begin{equation*}
-\frac{1}{\rho(y)-a}\frac{\|P_S r(y)\|^2}{\|y\|^2}\leq
\rho(y)-\lambda\leq\frac{1}{b-\rho(y)}\frac{\|P_S r(y)\|^2}{\|y\|^2},
\end{equation*}
which, after introducing $\delta=\min_{\eta\in\{\Sigma(A)\backslash\{\lambda\}\}}|\eta-\rho(y)|$,  implies 
\begin{equation*}
|\lambda-\rho(y)|\leq \frac{1}{\delta}\frac{\|P_S r(y)\|^2}{\|y\|^2}.
\end{equation*}
Theorem \ref{lemma:projectorresidualsqu} also gives us the improved Krylov-Weinstein bound, 
\begin{equation*}
|\lambda-\rho(y)|\leq \frac{\|P_S r(y)\|}{\|y\|}.
\end{equation*}
Here, as in the original Krylov-Weinstein bound \eqref{eq:residual1}, only the existence of 
$\lambda\in\Sigma(A)$, which may or may not be an eigenvalue, is guaranteed. 

\begin{remark}\label{rem:V}
We note that 
$\|P_{S}r(y)\|\leq\|P_{V}r(y)\|\leq\|r(y)\|$, since $S\subseteq V\subset \H$.  
Thus, first, the classical a posteriori bounds now trivially follows from our bounds. 
Second, if the smallest quantity $\|P_{S}r(y)\|$ may not be readily accessible
in a practical situation, it can be bounded above by the following much simpler 
expression, $\|P_{V}r(y)\|$, which can still be dramatically smaller 
compared to the standard value $\|r(y)\|$ used in the Temple bound~\eqref{eqn:betaalphabound}. 
As an alternative for $V=U +\span\{y\}$ as defined in Lemma~\ref{lemma:technical}, 
we can take $V=U^\perp +\span\{y\}\supseteq S$.
\end{remark}

As an illustration of possible improvements in 
$\|P_{S}r(y)\|\leq\|P_{V}r(y)\|\leq\|r(y)\|$, 
let us again, as in Remark \ref{rem:unbounded}, 
consider the example from Davis and Kahan \cite{dk70} in $\H=l_2$. 
Let $\epsilon=1/2$,  
$y=(1,\epsilon,\epsilon^2,\ldots)^T\in l_2,$ and 
$A=\diag(1,\epsilon^{-1},\epsilon^{-2},\ldots)$.
We get $\rho(y)=1+\epsilon$ even though both the  
vector $Ay=(1,1,1,\ldots)^T$ and the residual $r(y)=Ay-\rho(y)y$ 
have infinite norms in $l_2$. Since $\|r(y)\|=\infty$,  neither 
Temple bound~\eqref{eqn:betaalphabound}, nor any other $\|r(y)\|$-based bound, can be used in this example, 
as  Davis and Kahan correctly point out in \cite[p. 42]{dk70}.

At the same time, let us choose the subspace $U$ 
as described in Theorem \ref{lemma:projectorresidualsqu}. Specifically, since 
$\epsilon=1/2$, we have $\rho(y)=1+\epsilon=1.5\in(1,2),$ thus 
$U^\perp$ is simply the span of the first coordinate sequence $e_1=(1,0,0,\ldots)^T$, 
which is the eigenvector of $A$ corresponding to the smallest eigenvalue $\rho(x)=1$. 
This gives $x=(I-P_U)y=e_1$ and Theorem \ref{lemma:projectorresidualsqu} turns into its 
particular case, Lemma \ref{lemma:projectorresidualsquare1}.
Since ${S}=\span\{x, y\}$, the sequence  $x=e_1$ is an eigenvector of both operators $A$ and $A_S$ 
corresponding to the same smallest eigenvalue $\alpha=\lambda_{\min}\left(A\right)=\lambda_{\min}\left(A_S\right)=1$.
Then the sequence $w=y-x=(0,\epsilon,\epsilon^2,\epsilon^3,\ldots)^T$ is the second eigenvector of $A_S$ 
corresponding to the eigenvalue $\rho(w)=\lambda_{\max}\left(A_S\right)=(1+\epsilon)/\epsilon=3$, while $\beta=1/\epsilon=2.$
Rather than explicitly calculating the sequence $P_Sr(y)$ and its $l_2$ norm, we can take a short cut using  
Lemma~\ref{lemma:generalresidual1},
\begin{eqnarray*}
\frac{\|P_{S}r\left(y\right)\|^2}{\|y\|^2}
&=&
\left[\lambda_{\max}\left(A_S\right)-\rho(y)\right]
\left[\rho(y)-\lambda_{\min}\left(A_S\right)\right]\\
&=&
\left[\frac{1+\epsilon}{\epsilon}-(1+\epsilon)\right] \left[1+\epsilon-1\right]\\
&=& 
1-\epsilon^2
=
\frac34,
\end{eqnarray*}
which gives the  right-hand side of our bound \eqref{eqn:genearalproresidualsqu}. Its left-hand side is 
\[
\left(\beta-\rho(y)\right)\left(\rho(y)-\alpha\right)
=
\left(\frac1\epsilon-(1+\epsilon)\right) \left((1+\epsilon)-1\right)=
1-\epsilon-\epsilon^2=\frac14.
\]
We conclude that our bound $1/4\leq3/4$ holds, while the Temple bound fails, giving just  
the trivial statement $1/4\leq\infty$.

By modifying the choice of the sequence $y$ in the example above, e.g.,\
by choosing  $y=(0,1,\epsilon,\epsilon^2,\ldots)^T$, one can construct similar situations, but where 
Theorem \ref{lemma:projectorresidualsqu} is different from Lemma \ref{lemma:projectorresidualsquare1}.

Let us finally comment that in this example our alternative choices for the subspace $V$, namely,  
$V=U +\span\{y\}$ as in Lemma \ref{lemma:technical} or 
 $V=U^\perp +\span\{y\}$ as in Remark \ref{rem:V},
give extremely different results. Evidently, the first choice leads to $V=H$ 
and gives no improvement compared to the broken Temple bound, 
while the second choice gives $V=S$ and works, as has already been shown above. 

This example demonstrates an opportunity to easily treat unbounded operators. The unbounded case 
naturally appears for partial differential operators in the $L_2$ norm; see, e.g.,\
\cite{ka09} for a similar approach to analyze the Finite Element Method.  

\section{Improving known error bounds for eigenvectors}\label{sec:vectors}
The main topic of this work is RQ bounds. However, having the machinery of 
identities already constructed in section \ref{section2}, 
we need only very little extra effort to 
revisit and improve some well known error bounds for eigenvectors, as a bonus.
The reader might have noticed a couple of identities, 
\eqref{eqn:generalresidual2} and \eqref{eqn:tanres}, 
that have not yet been used. For convenience, we repeat them here, changing the 
notation in \eqref{eqn:generalresidual2}, for the case where $S=\span\{x,y\}$, 
the vector $x$ is an eigenvector of $A$ (and, thus, $A_S$) with the eigenvalue $\lambda=\rho(x)$, 
the scalar $\eta\neq\lambda$ denotes the other eigenvalue of $A_S$, 
and $\theta=\angle\left\{x,y\right\}$:
\begin{eqnarray}\label{eqn:sin2theta} 
\sin\left(2\theta\right)
&=&\frac{2}{\left|\lambda-\eta\right|}\frac{\|P_{S}r\left(y\right)\|}{\|y\|}
\end{eqnarray}
and, if $\theta<\pi/2,$
\begin{eqnarray} \label{eqn:tantheta} 
\tan\left(\theta\right)&=&\frac1{\left|\eta-\rho(y)\right|}
\frac{\|P_{S}r\left(y\right)\|}{\|y\|}.
\end{eqnarray}

In order to use identities \eqref{eqn:sin2theta} and \eqref{eqn:tantheta}, 
one needs to bound the scalar $\eta$. It is easily possible in the special case 
where $\lambda=\lambda_{\min}\left(A\right)$ or $\lambda=\lambda_{\max}\left(A\right)$.
The latter choice is reduced to the former one by substituting $-A$ for $A$.
We can use the arguments from the proof of Lemma \ref{lemma:projectorresidualsquare1} to handle 
the case $\lambda=\lambda_{\min}\left(A\right)$. 
Let the assumptions of Lemma \ref{lemma:projectorresidualsquare1} be satisfied, i.e., 
let $\alpha=\lambda_{\min}\left(A\right)$ be an isolated point of the spectrum 
of $A$ and $\beta>\alpha$ be the nearest larger point of the spectrum 
of $A$. Let $\X$ denote the complete eigenspace corresponding to 
the eigenvalue $\alpha.$ For a given vector $y\neq0$, such that $\alpha<\rho(y)<\beta$, we define 
$x=P_{\X}y$, where $P_{\X}$ is the orthogonal projector to the subspace $\X$; and we denote $S=\span\{x,y\}$. 
Then $\lambda=\rho(x)=\alpha$ and
$\beta\leq\eta$, directly by the proof of Lemma \ref{lemma:projectorresidualsquare1}. 
Substituting $\beta\leq\eta$ in \eqref{eqn:sin2theta} and \eqref{eqn:tantheta}, 
we obtain the following improvements of the corresponding bounds from~\cite[p. 11]{dk70},
\begin{eqnarray*} 
\sin\left(2\theta\right)
&\leq&\frac{2}{\beta-\lambda}\frac{\|P_{S}r\left(y\right)\|}{\|y\|}
\end{eqnarray*}
and, if $\theta<\pi/2,$ from~\cite[p. 11]{dk70} (see a single-vector version in \cite[Corollary 11.7.1]{parlett80}),
\begin{eqnarray*}
\tan\left(\theta\right)&\leq&\frac1{\beta-\rho(y)}
\frac{\|P_{S}r\left(y\right)\|}{\|y\|}.
\end{eqnarray*}

The example from Davis and Kahan \cite[p. 42]{dk70}, discussed above, is also applicable  
here, to observe that the use of the norm of the projected residual $\|P_{S}r\left(y\right)\|$ in our bounds 
can bring a dramatic improvement compared to the traditional term $\|r\left(y\right)\|$. 

One final observation concerns the classical $\sin(\theta)$ bound, see, e.g., 
\cite[p. 10]{dk70} and, for a single-vector version that we use here, 
 \cite[Corollary 6.22]{Chatelin} or \cite[Theorem 3.9]{saas92}, 
\begin{eqnarray} \label{eqn:sintheta} 
\sin\left(\theta\right)&\leq&\frac1{\delta}
\frac{\|r\left(y\right)\|}{\|y\|},
\end{eqnarray}
where  $\delta=\min_{\eta\in\{\Sigma(A)\backslash\{\lambda\}\}}|\eta-\rho(y)|$.
Interestingly, an attempt of a naive improvement of \eqref{eqn:sintheta}, where $\|P_{S}r\left(y\right)\|$ simply substitutes 
for  $\|r\left(y\right)\|$ without any other changes, fails as the following example demonstrates.  
Let $y=(1,1,1)^T$ and $A=\diag(1,0,-1)$. We get $Ay=(1,0,-1)^T$ and $\rho(y)=0$, so  $r(y)=Ay$.  
Let $x=(0,1,0)^T$ with $\lambda=\rho(x)=0$. The residual  $r(y)=(1,0,-1)^T$ is perpendicular both to 
$x$ and $y$, thus $P_{S}r\left(y\right)=0,$ while $\delta=1$ and $\sin^2(\theta)=2/3.$ At the same time, 
$\|r(y)\|^2/\|y\|^2=2/3$, i.e.,\ the original $\sin(\theta)$ bound \eqref{eqn:sintheta} holds fine. 
The moral of this observation is that one cannot just blindly substitute  $\|P_{S}r\left(y\right)\|$  
for  $\|r\left(y\right)\|$ and hope for improvement without breaking the original statement. 


\section*{Conclusion}
We demonstrate the fundamental nature of our new  RQ identities and concise inequalities, 
by using them to derive, in a unifying manner, and improve several known eigenvalue error bounds, 
including the famous Temple bound.  The next natural step is to attempt to extend our single-vector RQ results 
to matrix RQ, thus hoping to simplify the derivation and to improve some classical results for 
the Rayleigh-Ritz method. We conjecture, by analogy with the approach of this paper, 
that the key to such a development is a careful sharp analysis of sensitivity of 
the Ritz values with respect to variations in the trial subspace in the Rayleigh-Ritz method. 

\section*{Acknowledgments}
We thank Ilse Ipsen, Dianne O'Leary, and anonymous referees for constructive comments that 
helped us to improve the presentation. We are grateful to SIMAX Associate Editor Michiel Hochstenbach 
for handling our paper. 

\vskip12pt


\def\refname{
\end{document}